\documentclass{amsart}

\usepackage{amssymb}
\usepackage{amsfonts}
\usepackage{geometry}
\usepackage{hyperref}
\usepackage{graphicx}
\usepackage{chngcntr}

\newcommand{\R}{\mathbb{R}}

\newcommand{\ve}{\varepsilon}
\newcommand{\lv}{\left\vert}
\newcommand{\rv}{\right\vert}
\newcommand{\func}{\mathrm}

\newcommand{\N}{\mathbb{N}}

\DeclareMathOperator{\supp}{supp}

\newtheorem{theorem}{Theorem}[section]
\newtheorem{lemma}[theorem]{Lemma}
\newtheorem{prop}[theorem]{Proposition}
\newtheorem{corollary}[theorem]{Corollary}

\theoremstyle{definition}
\newtheorem{definition}[theorem]{Definition}

\theoremstyle{remark}
\newtheorem{remark}[theorem]{Remark}

\numberwithin{equation}{section}
\usepackage{enumerate}

\begin{document}

\title[Degenerate elliptic equations with $\Phi$-admissible weights]{Degenerate elliptic equations with $\Phi$-admissible weights}


\author[Korobenko]{Lyudmila Korobenko}
\address{Reed College\\
Portland, Oregon\\
korobenko@reed.edu}


\date{}

\dedicatory{}

\begin{abstract} We develop regularity theory for degenerate elliptic equations with the degeneracy controlled by a weight. More precisely, we show local boundedness and continuity of weak solutions under the assumption of a weighted Orlicz-Sobolev and Poincar\'{e} inequalities. The proof relies on a modified DeGiorgi iteration scheme, developed in \cite{KRSSh}. The Orlicz-Sobolev inequality we assume here is much weaker than the classical $(2\sigma,2)$ Sobolev inequality with $\sigma>1$ which is typically used in the DeGiorgi or Moser iteration.

\end{abstract}

\maketitle

\section{Introduction}\label{sec:intro}
The main focus of this paper is a degenerate elliptic operator of the form 
\begin{equation}\label{def:op}
Lu=\nabla^{\func{tr}} A(x)\nabla,
\end{equation}
where $A$ is an $n\times n$ symmetric nonnegative semidefinite matrix with bounded coefficients. 
Regularity theory for such operators has a long history going back to Nash \cite{Nash}, DeGiorgi \cite{DeG}, and Moser \cite{Mos}, who considered the case of an elliptic matrix $A$ and proved H\"{o}lder regularity of weak solutions to $Lu=0$. Since then it has been further developed and generalized in a numerous different ways, and we only mention a few most relevant works. In 1982 Fabes, Kenig, and Serapoini showed \cite{FKS} that the Moser iteration scheme can be implemented in the degenerate case, as long as the weight controlling the degeneracy satisfies the following four conditions with $p=2$: (1) the doubling property; (2) the uniqueness of the weak gradient; (3) the $(p,p)$ Poincar\'{e} inequality; (4) the $(q,p)$-Sobolev inequality with $q>p$. Weights satisfying (1)-(4) were called $p$-admissible in \cite{HKM} and a rich potential theory was developed for them, see \cite{BB11}. The interplay between these four conditions received considerable attention, see e.g. \cite{HK95, SC, KMR}. Shortly after Fabes, Kenig, and Serapioni's work, Franchi and Lanconelli \cite{FL} applied the Moser technique in the case when the degeneracy is not controlled by a weight, but is such that the necessary ingredients can be obtained for the Carnot-Carath\'{e}odory metric space instead. After these two pioneering works, the idea has been extended by many authors to more difficult situations, see e.g. \cite{GL, SW, C-UR, C-UMR}.

In most literature, the doubling condition is central in both the
process of performing the Moser iteration, and in the proof of Sobolev inequality
itself. This condition provides a homogeneous space structure, which makes it
possible to adapt many classical tools available in the Euclidean space \cite{Haj, Shan}. Moreover, it turns out to be necessary for the $(q,p)$ Sobolev inequality with $q>p$, \cite{KMR}.
However, such an assumption greatly restricts the type of degeneracy that the matrix $A$ is allowed to have. For example, a simple two dimensional matrix of the form
\[
A(x,y)=\mathrm{diag}\{1,f^2(x)\}
\]
considered in \cite{KRSSh}, does not produce a doubling metric measure space if $f$ vanishes to an infinite order, and therefore the Sobolev inequality is no longer available. A weaker Orlicz Sobolev inequality can be used instead, which requires a significant alteration of the classical DeGiorgi iteration scheme, see \cite{KRSSh} for details.

In this paper we generalize the abstract theory developed in \cite{KRSSh} in two ways. First, we work in a metric measure space, which allows for degenerate matrices controlled by a weight as in \cite{FKS, CW, FGW}, as well as one constant eigenvalue as in \cite{FL, SW,  KRSSh}, see also \cite{Fr}. Second, we assume an Orlicz Sobolev inequality with a very weak logarithmic gain. This seems to be the weakest Sobolev type inequality ever used in the Moser or DeGiorgi iteration scheme.  The proofs follow the general blueprint of \cite[Chapters 4, 6]{KRSSh}. However, significant modifications of the argument had to be made to account for (a) weaker $L^2$ versions of Sobolev and Poincar\'{e} inequalities are now assumed instead of $L^1$ versions; and (b) a more general family of Young functions is considered.

The paper is organized as follows. In Section \ref{sec:main} we give relevant definitions and state our main results. Section \ref{sec:fam} is devoted to the description of special families of Young functions and Lipschitz cutoff functions used in the paper. The proofs of the main results, Theorems \ref{DG} and \ref{thm:continuity}, are contained in Section \ref{sec:proofs}. Section \ref{sec:poinc} provides some sufficient conditions for Sobolev and Poincar\'{e} inequalities, as well as the admissibility condition for the right hand side. Finally, in Section \ref{sec:ex} we give examples of operators satisfying sufficient conditions in Theorems \ref{DG} and \ref{thm:continuity}.

\section{Main results}\label{sec:main}
We now briefly describe our main results. 
As in \cite{KRSSh} we consider the divergence form equation 
\begin{equation}\label{eq_0}
Lu=\nabla ^{\func{tr}}A\left( x\right) \nabla u=\phi ,\ \ \ \ \ x\in
\Omega ,
\end{equation}%
where $A$ is a nonnegative semidefinite $n\times n$ symmetric matrix satisfying for some constant $k>0$ and a weight $w$
\begin{equation}\label{A-cond}
\left\vert A(x) \xi\right\vert\leq kw(x)\left\vert\xi\right\vert,\quad \text{a.e.}\   x\in\Omega,\  \forall \xi\in\R^n.
\end{equation}
Suppose that $A(x)\approx B\left( x\right) ^{\func{tr}}B\left( x\right) $ where $B\left( x\right) $ is a Lipschitz continuous $%
n\times n$ real-valued matrix defined for $x\in \Omega $, i.e. there exist $c,C>0$ s.t. 
\begin{equation}
c\,\xi ^{\func{tr}}A(x)\xi \leq \left\vert B\left( x\right) \xi\right\vert^{2} \leq C\,\xi ^{\func{tr}}A(x)\xi
\ .  \label{struc_0}
\end{equation}%
We define the $A$%
-gradient by%
\begin{equation}
\nabla _{A}=B\left( x\right) \nabla \ ,  \label{def A grad}
\end{equation}%
and the associated degenerate Sobolev space $W_{A}^{1,2}\left( \Omega
\right) $ to have norm%
\begin{equation*}
\left\Vert v\right\Vert _{W_{A}^{1,2}}\equiv \sqrt{\int_{\Omega }\left(
	\left\vert v\right\vert ^{2}+\nabla v^{\func{tr}}A\nabla v\right) }\approx \sqrt{%
	\int_{\Omega }\left( \left\vert v\right\vert ^{2}+\left\vert \nabla
	_{A}v\right\vert ^{2}\right) }.
\end{equation*}

To define weak solutions to (\ref{eq_0}) we need to assume that the right hand side $\phi$ of (\ref{eq_0}) is $A$-admissible, which is defined as follows. 
\begin{definition}
	\label{def A admiss new}Let $\Omega $ be a bounded domain in $\mathbb{R}^{n}$%
	. Fix $x\in \Omega $ and $\rho >0$. We say $\phi $ is $A$\emph{-admissible}
	at $\left( x,\rho \right) $ if
	\begin{equation*}
	\Vert \phi \Vert _{X\left( B\left( x,\rho \right) \right) }\equiv \sup_{v\in
		\left( W_{A}^{1,2}\right) _{0}(B\left( x,\rho \right) )}w\left(\{y\in B\left( x,\rho \right)\colon v(y)\neq 0\}\right)^{-\frac{1}{2}}\frac{\left\vert\int_{B\left(
			x,\rho \right) } v\phi  \,dy\right\vert}{\Vert \nabla _{A}v\Vert_{L^{2}(B\left( x,\rho \right))}}<\infty .
	\end{equation*}
	We say $\phi $ is $A$\emph{-admissible} in $\Omega$, written $\phi\in X(\Omega)$, if $\phi$ is $A$\emph{-admissible}
	at $\left( x,\rho \right) $ for all $B(x,\rho)$ contained in $\Omega$.
\end{definition}

\begin{remark}
	For a simple sufficient condition of $A$-admissibility see Section \ref{sec:adm}
\end{remark}

\begin{definition}
	Let $\Omega $ be a bounded domain in $\mathbb{R}^{n}$ with $A$ as above. Assume that $\phi \in L_{\func{loc}}^{2}\left( \Omega
	\right) \cup X\left( \Omega \right) $. We say that $u\in W_{A}^{1,2}\left(
	\Omega \right) $ is a \emph{weak solution} to $Lu=\phi $ provided%
	\begin{equation*}
	-\int_{\Omega }\nabla g\left( x\right) ^{\func{tr}}A\left(
	x\right) \nabla u=\int_{\Omega }\phi g
	\end{equation*}%
	for all nonnegative $g\in \left( W_{A}^{1,2}\right) _{0}\left( \Omega \right) $, where $%
	\left( W_{A}^{1,2}\right) _{0}\left( \Omega \right) $ denotes the closure in 
	$W_{A}^{1,2}\left( \Omega \right) $ of the subspace of Lipschitz continuous
	functions with compact support in $\Omega $. Weak sub- and super- solutions are defined by replacing $=$ with $\geq$ and $\leq$ respectively in the display above.
\end{definition}

We now define Orlicz spaces, to state the weak version of the Sobolev inequality used in the proof.
Let $\Phi :\left[ 0,\infty \right)
\rightarrow \left[ 0,\infty \right) $ be a Young function, which for our
purposes is a convex piecewise differentiable (meaning there are at most
finitely many points where the derivative of $\Phi $ may fail to exist, but
right and left hand derivatives exist everywhere) function such that $\Phi
\left( 0\right) =0$ and

\begin{equation*}
\frac{\Phi \left( x\right) }{x}\rightarrow \infty \text{ as }x\rightarrow
\infty.
\end{equation*}%
Let $\nu $ be a $\sigma $-finite measure on a set $\Omega\subset \R^n$ and $L_{\ast }^{\Phi }$ be the set of measurable functions $f:\Omega\rightarrow 
\mathbb{R}$ such that $ 	\int_{\Omega}\Phi \left( \left\vert f\right\vert \right) d\nu<\infty$.
We define $L^{\Phi }=L^{\Phi }_{\nu}$ to be the linear span of $L_{\ast
}^{\Phi }$, and then define%
\begin{equation}\label{norm_def}
\left\Vert f\right\Vert _{L^{\Phi }_{\nu}\left(\Omega \right) }\equiv \inf \left\{
k\in \left( 0,\infty \right) :\int_{\Omega}\Phi \left( \frac{\left\vert
	f\right\vert }{k}\right) d\nu \leq 1\right\} .
\end{equation}%
The Banach space $L^{\Phi }_{\nu}\left(\Omega \right) $ is precisely the space of
measurable functions $f$ for which the norm $\left\Vert f\right\Vert
_{L^{\Phi }_{\nu}\left(\Omega \right) }$ is finite.  For an extensive theory of Orlicz spaces we refer to \cite{RaoRen}.

Finally, let $d$ be a metric on $\Omega$, which in applications is usually a Carnot-Carath\'{e}odory, or subunit, metric associated to the operator, see e.g. \cite{FePh}. For our abstract results, however, we do not need to make this assumption. Let $B=B(y,r)$ be any metric ball in $\Omega$.
Here are our hypotheses for the metric measure space $(\Omega, d,\mu)$, where $d\mu(x)=w(x)dx$. 
\begin{enumerate}[(H1)]
	\item Proportional vanishing Poincar\'{e} inequality
	\begin{equation}\label{1-1poinc}
	\int_{B}|v|^2 d\mu\leq Cr^2\int_{\frac{3}{2}B}\left\vert \nabla_{A} v\right\vert^2dx,\ \ \ \ \ v\in Lip\left( \frac{3}{2}B\right) ,
	\end{equation}%
	provided that $v$ vanishes on a subset $E\subseteq B$ satisfying $w(E)\geq \frac{1}{2}w(B)$.
	\item Orlicz-Sobolev inequality
	\begin{equation}\label{orl-sob-w}
	\left\Vert v\right\Vert _{L^{\Phi }_{\mu_{B}}}\leq C\frac{\varphi \left(
		r\left( B\right) \right) }{w(B)^{\frac{1}{2}}}\left\Vert \nabla_{A} v\right\Vert _{L^{2} },\ \ \ \ \ v\in Lip_{c}\left( B\right) ,
	\end{equation}%
	where $\varphi$ is an increasing function defined on some interval $[0,R]$ satisfying $\varphi(r)\geq r$ $\forall r\in[0,R]$, $\varphi(0)=0$, and is referred to as the ``superradius''. Here
	$\mu_B$ is the normalized weighted measure defined as
	\[
	d\mu_B(x)\equiv\frac{w(x)dx}{w(B)},
	\]
	where 
	\[
	w(B)=\int_{B}w(x)dx,
	\]
	and $B\subset \Omega\subset\R^n$ is a fixed metric ball.
	
	\item Existence of the adapted sequence of Lipschitz cutoff functions $\{\psi_j\}_{j=1}^{\infty}$ as in Definition \ref{def_cutoff}.
\end{enumerate}	

We can now state our main results.
\begin{theorem}[Local boundedness]
	\label{DG}Assume that for every ball $B\subset \Omega$ the Orlicz Sobolev norm inequality (\ref{orl-sob-w}) holds
	with $\Phi =\Phi _{k,N}$ as in (\ref{phi-def}) for some $k=1,2,\dots$, $N>1$, and superradius $\varphi \left(
	r\right) $. In addition, assume that there exists an adapted sequence of Lipschitz cutoff functions as in Definition \ref{def_cutoff}. Suppose further that $\phi$ is $A$-admissible as in Definition \ref{def A admiss new}.
	
	\begin{enumerate}
		\item Then every weak subsolution $u$ to $Lu=\phi $ in $\Omega $,
		with $L$ as in (\ref{eq_0}), (\ref{A-cond}), satisfies the \emph{inner ball inequality}%
		\begin{equation}
		\left\Vert u_{+}\right\Vert _{L^{\infty }\left(\frac{1}{2}B\right)}\leq A_{k,N}(r)\left( 
		\frac{1}{w(B) }\int_{B}u_{+}^{2}w\right) ^{\frac{1}{2}%
		}+\left\Vert \phi\right\Vert _{X},  \label{Inner ball inequ}
		\end{equation}%
		\begin{equation}
		\text{where }A_{k,N}(r)\equiv C_{1}\left[\exp^{(k)} \left(
		C_{2}\left( \frac{\varphi \left( r\right) }{r}\right) ^{\frac{2}{N-1}}\right)\right]^{\frac{1}{2}},  \label{def AN}
		\end{equation}%
		for every ball $B\subset \Omega $ with radius $r$. Here the constants $C_{1}$ and $C_{2}$ depend on $N$ and $k$ but not on $r$.
		
		
		\item In particular, weak solutions are locally bounded in $\Omega $.
	\end{enumerate}
\end{theorem}

\begin{theorem}[Continuity]\label{thm:continuity}
	Suppose that for every $x\in\Omega$, there exists $R>0$ s.t. for every $0<r<R$ the proportional vanishing Poincar\'{e} inequality (\ref{1-1poinc}) holds in $B=B(x,r)$. Further assume that the Orlicz Sobolev norm inequality (\ref{orl-sob-w}) holds
	with $\Phi =\Phi _{k,N}$ as in (\ref{phi-def}) for some $k=1,2,\dots$, $N>1$, and superradius $\varphi \left(
	r\right) $. In addition, assume that there exists an adapted sequence of Lipschitz cutoff functions as in Definition \ref{def_cutoff}. Finally, suppose that there exist $r_0>0$, $\alpha\in(0,(N-1)/(2N))$, and $C>0$ s.t. the superradius $\varphi$ satisfies
	\begin{equation}\label{phi-cond}
	\varphi(r)\leq Cr\left[\ln^{(k+2)}\left(\frac{1}{r}\right)\right]^{\alpha N},\quad\forall r\in(0,r_0).
	\end{equation}
	Then if $u$ is a weak solution to $L
	u=0$ in $\Omega$, we have $u\in C\left( \Omega\right) $.
\end{theorem}

The above results justify the following definition
\begin{definition}\label{def:admiss}
	The weight $w$ is \emph{$\Phi$-admissible} if the following Orlicz-Sobolev inequality holds for all $ v \in \left(W^{1,2}_{w}\right)_{0}(B)$
	\begin{equation}\label{weight-orl-sob}
	\left\Vert v\right\Vert _{L^{\Phi }_{\mu_{B}}}\leq C\varphi \left(
	r\left( B\right) \right) \left\Vert \nabla v\right\Vert _{L^{2}_{\mu_{B}} }
	\end{equation}%
	holds with superradius $\varphi(r)$.
	If, in addition, the weighted Poincar\'{e} inequality 
	\begin{equation}\label{weight-poinc}
	\int_{B}|v-v_B|^2 d\mu\leq Cr\int_{\frac{3}{2}B}\left\vert \nabla v\right\vert^2\,d\mu,
	\end{equation}
	where $v_B=\frac{1}{w(B)}\int_{B}vd\mu$
	(\ref{1-1poinc}) holds for all $v \in W^{1,2}_{w}(B)$, the weight $w$ is called \emph{strongly $\Phi$-admissible}.
\end{definition}

Indeed, if the matrix $A$ has a degeneracy controlled by weight $w$, i.e. it satisfies (\ref{A-weighted}), then (\ref{weight-orl-sob}) implies (\ref{orl-sob-w}). Moreover, in this case it follows from Section \ref{sec:poinc} that (\ref{weight-poinc}) implies (\ref{1-1poinc}). Theorems \ref{DG} and \ref{thm:continuity} then imply local boundedness and continuity of weak solutions, thus justifying Definition \ref{def:admiss}. 
\subsection{Relation to previous results}

We establish sufficient conditions for local boundedness and continuity of weak solutions. Despite the seemingly abundant number of such results, our theorems stand out in two ways. (1) We require the weakest version of a Sobolev inequality considered elsewhere in the literature. This makes the implementation of the DeGiorgi iteration scheme quite technical, and requires careful estimates at every step of the proof. As a payoff it possibly allows for more degenerate operators. (2) Our theory generalizes several existing results (see Section \ref{sec:ex}) , demonstrating that it is not vacuous.

\section{Families of Young functions}\label{sec:fam}

We now describe a two-parameter family of Young functions, $\{\Phi_{k,N}\}$, which gives rise to Orlicz-Sobolev norms used in our adaptation of DeGiorgi iteration. Let $N> 1$, $k=1,2,\dots$, and define
\begin{equation}\label{phi-def}
\Phi_{k,N}(t)\equiv \left\{ 
\begin{array}{ccc}
t^2\cdot\left(\ln t^2\right)^{2}\cdot \left(\ln\ln t^2\right)^{2}\cdot\dots\cdot(\ln^{(k)} t^2)^{2N} & \text{ if } & t^2\geq E=E_{k,N}=\exp^{(k)}(2N) \\ 
t^2\cdot\left(\ln E\right)^{2}\cdot \left(\ln\ln E\right)^{2}\cdot\dots\cdot(\ln^{(k)} E)^{2N} & \text{ if } & 0\leq t^2\leq E=E_{k,N}=\exp^{(k)}(2N)%
\end{array}%
\right. , 
\end{equation} 
where we use the notation $f^{(k)}$ for the $k^{th}$ composition of $f$, $f\circ f\circ\dots\circ f$.
\begin{remark}
	The family $\{\Phi_{k,N}\}$ is a generalized $L^2$ version of the family $\{\Phi_N\}$ considered in \cite{KRSSh}. More precisely,
	\[
	\Phi_N(t)\approx\sqrt{\Phi_{1,N}(t)},
	\]
	where the approximate equality is due to the fact that we replaced the argument $t$ in the logarithms by $t^2$. This is done to reduce notational clutter later when considering the function $\Phi^{0}_{k,N}(t)=\Phi_{k,N}(\sqrt{t})$, and does not change the behavior of $\Phi_{k,N}(t)$ for large $t$.
\end{remark}
We now define $\Phi^{0}_{k,N}(t)=\Phi_{k,N}(\sqrt{t})$, i.e.
\begin{equation}\label{phi0-def}
\Phi^{0}_{k,N}(t)\equiv \left\{ 
\begin{array}{ccc}
t\cdot\left(\ln t\right)^{2}\cdot \left(\ln\ln t\right)^{2}\cdot\dots\cdot(\ln^{(k)} t)^{2N} & \text{ if } & t\geq E=E_{k,N}=\exp^{(k)}(2N) \\ 
t\cdot\left(\ln E\right)^{2}\cdot \left(\ln\ln E\right)^{2}\cdot\dots\cdot(\ln^{(k)} E)^{2N} & \text{ if } & 0\leq t\leq E=E_{k,N}=\exp^{(k)}(2N)%
\end{array}%
\right. , 
\end{equation} 
and establish some useful properties of $\Phi_{0}\equiv\Phi^{0}_{k,N}$.
First note that for each $k$ and $N$, $\Phi_0$ is convex, increasing, $\Phi_{0}(0)=0$ and $\frac{\Phi_{0}(t)}{t}\to \infty$ as $t\to\infty$. Next, since $\Phi_{0}$ is the linear extension of a submultiplicative Young function defined on $[E,\infty)$, it is submultiplicative on $[0,\infty)$ provided it is submultiplicative on $[E,\infty)$ (see Chapter 4 of \cite{KRSSh}). Let $s,t\geq \exp^{(k)}(2N)$ and $N\geq 1$, then using the fact that $a+b\leq ab$ if $a,b\geq 2$ we have
\begin{align*}
\ln(st)=\ln s +\ln t&\leq [\ln s][\ln t]\\
\ln^{(2)}(st)=\ln(\ln s +\ln t)\leq \ln^{(2)} s +\ln^{(2)} t&\leq \left[\ln^{(2)}s\right]\left[\ln^{(2)}t\right]\\
&\dots\\
\ln^{(k)}(st)&\leq \left[\ln^{(k)}s\right]\left[\ln^{(k)}t\right].
\end{align*}
This gives 
\begin{align*}
&(st)\cdot\left(\ln (st)\right)^{2}\cdot \left(\ln\ln (st)\right)^{2}\cdot\dots\cdot(\ln^{(k)} (st))^{2N}\\
&\quad\leq \left[t\cdot\left(\ln t\right)^{2}\cdot \left(\ln\ln t\right)^{2}\cdot\dots\cdot(\ln^{(k)} t)^{2N}\right]\left[s\cdot\left(\ln s\right)^{2}\cdot \left(\ln\ln s\right)^{2}\cdot\dots\cdot(\ln^{(k)} s)^{2N}\right],
\end{align*}
so $\Phi^{0}_{k,N}$ is submultiplicative on $[E,\infty)$, and therefore on $[0,\infty)$.\\
Next, we derive estimates on the conjugate of $\Phi_{0}$ we will need later in the implementation of the DeGiorgi iteration. 
By \cite[Lemma 47]{S20} we have the following inequality for the conjugate Young function $\widetilde{\Phi}_{0}$
\[
\widetilde{\Phi}_{0}\left(\frac{\Phi_0(t)}{t}\right)\leq \Phi_{0}(t)\leq \widetilde{\Phi}_{0}\left(\frac{2\Phi_0(t)}{t}\right)\quad \forall t>0.
\]
This immediately gives
\[
\widetilde{\Phi}_{0}^{-1}\left(\Phi_0(t)\right)\geq \frac{\Phi_0(t)}{t},
\]
and thus
\begin{equation}\label{phi-0-est}
\widetilde{\Phi}_{0}^{-1}\left(s\right)\geq \frac{s}{\Phi_{0}^{-1}(s)}.
\end{equation}
We now estimate $\Phi_{0}^{-1}(s)$ for $s\geq \Phi_{0}(E)$. Let $t\geq E$, $s=\Phi_0(t)$, we have
\begin{align*}
s&=	t\cdot\left(\ln t\right)^{2}\cdot \left(\ln\ln t\right)^{2}\cdot\dots\cdot(\ln^{(k)} t)^{2N} \\
\ln s&=\ln t+2\ln\ln t+\dots +2N\ln^{(k+1)} t\\
&\leq 2(k+N)\ln t.
\end{align*}
This gives
\[
\ln t\geq \frac{\ln s}{C_{k,N}},\quad \text{and}\quad  \ln\ln s\leq \ln C_{k,N}+\ln\ln t.
\]
Since $t\geq E=\exp^{(k)}(2N)$, the last inequality can be rewritten as
\[
\ln\ln s\leq  \tilde{C}_{k,N}\ln\ln t,
\]
for some constant $\tilde{C}_{k,N}$ depending only on $k$ and $N$. Similarly, we obtain with some other constant $C_{k,N}$
\[
\ln^{(j)}s\leq C_{k,N}\ln^{(j)}t,\  \  \forall 1\leq j\leq k.
\]
Altogether this gives
\[
t=\frac{s}{\left(\ln t\right)^{2}\cdot \left(\ln\ln t\right)^{2}\cdot\dots\cdot(\ln^{(k)} t)^{2N}}\leq \frac{C s}{\left(\ln s\right)^{2}\cdot \left(\ln\ln s\right)^{2}\cdot\dots\cdot(\ln^{(k)} s)^{2N}},
\]
so that
\[
\Phi_{0}^{-1}(s)\leq\frac{C s}{\left(\ln s\right)^{2}\cdot \left(\ln\ln s\right)^{2}\cdot\dots\cdot(\ln^{(k)} s)^{2N}}\quad\text{for}\  s\geq \Phi_0(E),
\]
where the constant $C$ depends only on $k$ and $N$. Substituting into (\ref{phi-0-est}) we obtain 
\begin{equation}\label{use}
\widetilde{\Phi}_{0}^{-1}\left(s\right)\geq C_{k,N}\left(\ln s\right)^{2}\cdot \left(\ln\ln s\right)^{2}\cdot\dots\cdot(\ln^{(k)} s)^{2N}
\end{equation}
for $s\geq \Phi_0(E)=\exp^{(k)}(2N)\left(\exp^{(k-1)}(2N)\right)^2\dots \left(\exp(2N)\right)^2(2N)^{2N}$.

\subsection{Families of cutoff functions}

In this section we describe the family of cutoff functions to be used in our adaptation of the DeGiorgi iteration scheme. When a classical $(q,p)$ Sobolev inequality with $q>p$ is available, one can work with a sequence of balls $\frac{1}{2}B\subset \dots\subset B_{j+1}\subset B_{j}\subset \dots\subset B$ with the radii
\[
r_1=r(B)=r,\  r_{j}=\frac{r}{2}+ \frac{r}{2^{j}},\   j=1,\dots.
\]
The corresponding sequence of cutoff functions $\{\psi_j\}$ can be chosen to satisfy
\[
\psi_j=1\  \text{on}\  B_{j+1},\  \psi_j=0\  \text{on}\  B_{j}^{c},\  	\left\Vert \nabla \psi _{j}\right\Vert _{\infty }\lesssim \frac{2^{j }}{r}.
\]
If one then uses this sequence in a Caccioppoli inequality like (\ref{Cacc_inhomog}), a factor of $2^j$ arises on the right hand side of that inequality. This factor can be handled by the Moser or DeGiorgi iteration, when the classical $(q,p)$ Sobolev inequality with $q>p$ is used. However, when working with a weaker, Orlicz-Sobolev inequality, this factor is too large for the iteration to converge. Roughly speaking, the smaller the Orlicz bump, the smaller the factor we can allow on the right hand side of the Caccioppoli inequality. The family of cutoff functions we construct below is adapted to the family of Young functions defined in (\ref{phi-def}). Fix $k\in\mathbb{N}$ and $N>1$, define the sequence of radii $\{r_j\}$ as follows
\begin{equation}\label{radii}
r_{1}=r,\  r_{\infty }\equiv
\lim_{j\rightarrow \infty }r_{j}=\frac{1}{2}r,\  r_j-r_{j+1}=\frac{cr}{\kappa_{k,N}(j)}\equiv\begin{cases}
\frac{cr}{j^2},\quad&\text{if}\  j\leq E=\exp^{(k)}(2N)\\
\frac{cr}{j\ln j\ln\ln j\dots(\ln^{(k-1)}j)^{\frac{N+1}{2}}},\quad&\text{if}\  j> E=\exp^{(k)}(2N),
\end{cases}
\end{equation}
where we introduced the notation
\begin{equation}\label{def:kappa}
\kappa_{k,N}(j)\equiv\begin{cases}
j^2,\quad&\text{if}\  j\leq E=\exp^{(k)}(2N)\\	
j\ln j\dots \ln^{(k-2)}j(\ln^{(k-1)}j)^{\frac{N+1}{2}},\quad&\text{if}\  j> E=\exp^{(k)}(2N).
\end{cases}
\end{equation}
Here the constant $c>0$ is uniquely determined by the condition $r_\infty=r/2$. The family $\{r_j\}$ is parametrized by $k$ and $N$, but we omit this dependence to simplify the notation.

\begin{definition}[Adapted sequence of accumulating Lipschitz functions]
	\label{def_cutoff}Let $\Omega $ be a bounded domain in $\mathbb{R}^{n}$. Fix $k\in\mathbb{N}$, $N>1$,
	$r>0$ and $x\in \Omega $. Let $\{r_j\}$ be a sequence of radii satisfying (\ref{radii}). We define an $\left(A,d,w;k,N\right) $-\emph{%
		adapted} sequence of Lipschitz cutoff functions $\left\{ \psi _{j}\right\}
	_{j=1}^{\infty }$ at $\left( x,r\right) $, along with sets $%
	B(x,r_{j})\supset \mathrm{supp}\psi _{j}$, to be a sequence satisfying $%
	\psi _{j}=1$ on $B(x,r_{j+1})$,  and 
	\[
	\left\Vert \frac{|\nabla_{A}
		\psi_{j}|}{w^{1/2}} \right\Vert _{L^{\infty }}\lesssim \frac{\kappa_{k,N}(j)}{r}	
	\]
	with $\nabla _{A}$ as in (\ref{def A grad}), $\kappa_{k,N}(j)$ as in (\ref{def:kappa}), and $w$ from (\ref{A-cond}).
\end{definition}

\section{Proofs of main results}\label{sec:proofs}

We first prove the following Caccioppoli inequality

\begin{prop}[Caccioppoli]\label{prop:cacc}
	Suppose
	for some ball $B\subset \Omega$, a constant $P>0$, and a nonnegative function $v\in L^{2}(\Omega)$
	there holds 
	\begin{equation}
	\left\Vert \phi \right\Vert _{X}\equiv 	\left\Vert \phi \right\Vert _{X(B)}\leq Pv(x),\quad \text{a.e.}\ x\in
	\{u>0\}\cap B.  \label{P_bound}
	\end{equation}%
	Then 
	\begin{equation}
	\int |\nabla_{A}(\psi u_{+})|^{2}dx\leq C\left( \left\Vert \frac{|\nabla_{A}
		\psi|^{2}}{w} \right\Vert _{L^{\infty }}+P^{2}\right) \int \left(
	u_{+}^{2}+v^{2}\right) wdx ,  \label{Cacc_inhomog}
	\end{equation}%
	for every weak subsolution $u$ of $Lu=\phi$ and $\psi\in Lip_{c}(B)$.
\end{prop}

\begin{proof}
	By definition of weak subsolution we have
	\[
	-\int \nabla g\cdot A\nabla u\geq \int \phi g
	\]
	for any nonnegative function $g\in W^{1,2}_0(B)$. Let $g=\psi^2u_{+}$, this gives
	\[
	-\int \nabla \left(\psi^2u_{+}\right)\cdot A\nabla u\geq \int \phi \psi^2u_{+}.
	\]
	Expanding the left hand side and rearranging we obtain using Definition \ref{def A admiss new}, condition (\ref{struc_0}), and Young's inequality
	\begin{align*}
	\int \lv\nabla_{A} \left(\psi u_{+}\right)\rv^{2}&\leq C	\int \nabla \left(\psi u_{+}\right)\cdot A\nabla  \left(\psi u_{+}\right)\leq C\int \left\vert\nabla_{A}\psi\right\vert^2u_{+}^{2} -C\int \phi \psi^2u_{+}\\
	&\leq C\int \frac{\left\vert\nabla_{A}\psi\right\vert^2}{w}u_{+}^{2}w +Cw\left(B\cap\{\psi^2 u_{+}\neq 0 \}\right)^{\frac{1}{2}}\lVert \phi\rVert_{X}\lVert \nabla_{A}\left(\psi^2 u_{+}\right)\rVert_{L^2}\\
	&\leq C\left\Vert \frac{|\nabla_{A}
		\psi|^{2}}{w} \right\Vert _{L^{\infty }}\int_{\supp \psi} u_{+}^2w+\frac{1}{\ve}w\left(B\cap\{u> 0 \}\right)\lVert \phi\rVert_{X}^{2}+\ve\int \lv \nabla_{A}\left(\psi^2 u_{+}\right)\rv^{2}
	\end{align*}
	for any $\ve>0$. Next, we have
	\begin{align*}
	\int \lv \nabla_{A}\left(\psi^2 u_{+}\right)\rv^{2}&=\int \lv \psi u_{+}\nabla_{A}\psi+\psi\nabla_{A}\left(\psi u_{+}\right)\rv^{2}\\
	&\leq  2\left\Vert \frac{|\nabla_{A}
		\psi|^{2}}{w} \right\Vert _{L^{\infty }}\int\psi^2u_{+}^{2}w+2\left\Vert\psi\right\Vert _{L^{\infty }}\int \lv\nabla_{A} \left(\psi u_{+}\right)\rv^{2}.
	\end{align*}
	Combining with the above gives
	\begin{equation}\label{eq:inter}
	\int \lv\nabla_{A} \left(\psi u_{+}\right)\rv^{2}\leq C\left\Vert \frac{|\nabla_{A}
		\psi|^{2}}{w} \right\Vert _{L^{\infty }}\int_{\supp \psi} u_{+}^2w+\frac{1}{\ve}w\left(B\cap\{u> 0 \}\right)\lVert \phi\rVert_{X}^{2}+\ve C\int \lv\nabla_{A} \left(\psi u_{+}\right)\rv^{2}.
	\end{equation}
	Choosing $\ve$ sufficiently small and absorbing the last term on the right hand side to the left gives
	\[
	\int \lv\nabla_{A} \left(\psi u_{+}\right)\rv^{2}\leq C\left\Vert \frac{|\nabla_{A}
		\psi|^{2}}{w} \right\Vert _{L^{\infty }}\int_{\supp \psi} u_{+}^2w+Cw\left(B\cap\{u> 0 \}\right)\lVert \phi\rVert_{X}^{2}.
	\]
	We now use (\ref{P_bound}) to obtain
	\[
	w\left(B\cap\{u> 0 \}\right)\lVert \phi\rVert_{X}^{2}=\int_{B\cap\{u> 0 \}}\lVert \phi\rVert_{X}^{2}w(x)dx\leq P^2\int_{B}v^2(x)w(x)dx,
	\]
	which combined with (\ref{eq:inter}) gives
	\[
	\int \lv\nabla_{A} \left(\psi u_{+}\right)\rv^{2}\leq C\left\Vert \frac{|\nabla_{A}
		\psi|^{2}}{w} \right\Vert _{L^{\infty }}\int_{\supp \psi} u_{+}^2w+CP^2\int_{B}v^2w.
	\]
	This shows (\ref{Cacc_inhomog}).
\end{proof}

\begin{remark}
	Note that (\ref{A-cond}) implies that
	\[
	\left\Vert \frac{|\nabla_{A}
		\psi|^{2}}{w} \right\Vert _{L^{\infty }}\lesssim \left\Vert \nabla
	\psi \right\Vert^{2} _{L^{\infty }}.
	\]
\end{remark}

\subsection{Local boundedness}

\begin{proof}[Proof of Theorem \ref{DG}]
	Without loss of generality, we may assume that $B=B\left( 0,r\right) $ is a
	ball centered at the origin with radius $r>0$. Let $\left\{ \psi
	_{j}\right\} _{j=1}^{\infty }$ be an \emph{adapted} sequence of Lipschitz
	cutoff functions at $\left( 0,r\right) $ as in Definition \ref{def_cutoff}. Following DeGiorgi (\cite{DeG}, see also 
	\cite{CaVa}) and \cite{KRSSh}, we consider the family of truncations 
	\begin{equation}
	\begin{split}\label{truncations}
	&u_{j}=(u-C_{j})_{+},\quad C_{j}=\tau \left\Vert \phi \right\Vert _{X}\left(
	1-c_j\right),\\
	&c_j=\begin{cases}
	c(j+1)^{-1}&\quad \text{if}\  1\leq j+1< E=\exp^{(k)}(2N)\\
	\left(\ln^{(k-1)}(j+1)\right) ^{-\frac{N-1}{2}} &\quad \text{if}\  j+1\geq E=\exp^{(k)}(2N) 
	\end{cases}
	\end{split},
	\end{equation}%
	where $c>0$ is a constant depending on $k$ and $N$ chosen to satisfy the condition $c_{j+1}<c_{j}$ for all $j$, and denote the $L^{2}$ norm of the truncation $u_{j}$ by 
	\begin{equation}
	U_{j}\equiv \int_{B_{j}}|u_{j}|^{2}d\mu_B ,  \label{def Uk}
	\end{equation}%
	where $d\mu_B =\frac{w(x)dx}{w(B(0,r))}=\frac{w(x)dx}{w(B) }$ and $w(B)=\int_{B}w(x)dx$ are
	independent of $k$. In $C_j$ we have introduced a parameter $\tau \geq 1$ that
	will be used later for rescaling. We will assume $||\phi||_X>0$, otherwise we replace it with a parameter $m>0$ and take the limit $m\to 0$ at the end of the argument.
	
	Using H\"{o}lder's inequality for Young functions we can write 
	\begin{equation}
	\int \left( \psi _{j+1}u_{j+1}\right) ^{2}d\mu \leq C||\left( \psi
	_{j+1}u_{j+1}\right) ^{2}||_{L^{\Phi_{0} }\left( B_{j};\mu_B \right) }\cdot
	||1||_{L^{\tilde{\Phi}_{0}}\left( \{\psi _{j+1}u_{j+1}>0\};\mu_B \right) }\ ,
	\label{hol}
	\end{equation}%
	where the norms are taken with respect to the measure $\mu_B $. Consider now a Young function $\Phi$ defined by 
	\[
	\Phi(t)=\Phi_0(t^2),
	\]
	with $\Phi_0$ defined in (\ref{phi0-def}), and note that if $u\in L_{w}^{\Phi}(B)$ then $u^2\in L_{w}^{\Phi_0}(B)$ and 
	\begin{equation}\label{phi_ineq}
	||u^2||_{L^{\Phi_0}(\mu)}\leq ||u||^{2}_{L^{\Phi}(\mu)}.
	\end{equation}
	Indeed, we have
	\[
	\int \Phi_0\left(\frac{u^2}{||u||^{2}_{L^{\Phi}(\mu)}}\right)d\mu = 	\int \Phi\left(\frac{u}{||u||_{L^{\Phi}(\mu)}}\right)d\mu \leq 1,
	\]
	which implies (\ref{phi_ineq}).
	
	Thus, for the first
	factor on the right of (\ref{hol}) we have, using (\ref{phi_ineq}) and the Orlicz Sobolev inequality (\ref{orl-sob-w}),
	\begin{align}\label{os_applied}
	||\left( \psi _{j+1}u_{j+1}\right) ^{2}||_{L^{\Phi_0 }\left( B_{j};\mu_B \right)} &\leq ||\left( \psi _{j+1}u_{j+1}\right) ||^{2}_{L^{\Phi }\left( B_{j};\mu_B \right)} \\
	&\leq C\frac{\varphi (r)^2}{w(B)}\int \left\vert \nabla_{A} \left( \psi
	_{j+1}u_{j+1}\right) \right\vert^{2}\ dx.\notag
	\end{align}
	
	We would now like to apply the Caccioppoli inequality (\ref{Cacc_inhomog})
	with an appropriate function $v$, and
	therefore we need to establish estimate (\ref{P_bound}). For that we observe
	that 
	\begin{eqnarray}
	u_{j+1} &>&0\Longrightarrow u>C_{j+1}=\tau \left\Vert \phi \right\Vert
	_{X}\left( 1-c_{j+1}\right)  \label{uk_support}
	\\
	&\Longrightarrow &u_{j}=\left( u-C_{j}\right) _{+}>\tau \left\Vert \phi
	\right\Vert _{X}\left[c_{j}-c_{j+1}\right] \notag.
	\end{eqnarray}%
	We now need to estimate $c_{j}-c_{j+1}$. Using the definition (\ref{truncations}) we first obtain for $0\leq j<E-1$
	\[
	c_{j}-c_{j+1}=\frac{c}{(j+1)(j+2)}\geq \frac{c}{(j+2)^2}.
	\]
	For $j\geq E$
	using the mean value theorem we have
	\begin{eqnarray*}
		c_{j}-c_{j+1}&=&	\left(\ln^{(k-1)}(j+1)\right) ^{-\frac{N-1}{2}} -\left(\ln^{(k-1)}(j+2)\right) ^{-\frac{N-1}{2}} \\
		&=&\frac{(N-1)}{2}\cdot\frac{1}{\theta\ln\theta\dots \ln^{(k-2)}\theta(\ln^{(k-1)}\theta)^{\frac{N+1}{2}}}\ \ \ \ \ \text{where }j+1\leq\theta\leq j+2 \notag\\
		&\geq &\frac{(N-1)}{2}\cdot\frac{1}{(j+2)\ln(j+2)\dots \ln^{(k-2)}(j+2)(\ln^{(k-1)}(j+2))^{\frac{N+1}{2}}}\notag.
	\end{eqnarray*}
	Finally, if $j+1\geq E$ and $j<E$, we have
	\[
	c_{j}-c_{j+1}\geq \frac{1}{2(j+1)}.	
	\]
	Altogether we conclude that for some constant $C=C(N)>0$, bounded uniformly for $N\geq 1$, there holds
	\begin{equation}\label{c est}
	c_{j}-c_{j+1}\geq \frac{(N-1)}{C\kappa_{j+2}},
	\end{equation}
	where $\kappa_{j+2}=\kappa_{k,N}(j+2)$ is defined in (\ref{def:kappa}).
	This implies 
	\begin{equation}\label{phi_x_bound}
	\left\Vert \phi \right\Vert _{X}\leq \frac{C}{\tau (N-1) \kappa_{j+2}}u_{j}\leq \frac{C}{(N-1) \kappa_{j+2}}u_{j},
	\end{equation}%
	which is (\ref{P_bound}) with $v=u_{j}$ and $P=\frac{C}{\tau (N-1) \kappa_{j+2}}$.
	Note that we have used our
	assumption that $\tau \geq 1$ in the display above.
	Thus by the Caccioppoli inequality (\ref{Cacc_inhomog}) and Definition \ref{def_cutoff} we have with $d\mu=w(x)dx$, 
	\begin{align*}
	\int \left\vert \nabla_{A} \left( \psi _{j+1}u_{j+1}\right) \right\vert
	^{2}\ dx
	&\leq C(N)\left( \left\Vert \frac{|\nabla_{A}
		\psi_{j+1}|^{2}}{w} \right\Vert _{L^{\infty }}+	\kappa_{j+2}^{2}\right)\int_{B_{j}}\left( u_{j+1}^{2}+u_{j}^{2}\right)
	d\mu \\
	&\leq \frac{C}{r^2}	\kappa_{j+2}^{2}
	\int_{B_{j}}u_{j}^{2}d\mu \ .
	\end{align*}%
	Combining with (\ref{os_applied}) we obtain%
	\begin{equation} \label{sob_cac} 
	||\left( \psi _{j+1}u_{j+1}\right) ^{2}||_{L^{\Phi }\left( B_{j};\mu_B \right)
	} \leq C\left[\frac{\varphi (r)}{r}\right]^2\kappa_{j+2}^{2}
	\int_{B_{j}}u_{j}^{2}d\mu_B \ .
	\end{equation}
	For the second factor in (\ref{hol}) we claim 
	\begin{equation}
	||1||_{L^{\tilde{\Phi}_0}\left( \{\psi _{j+1}u_{j+1}>0\}d\mu_B \right) }\leq
	\Gamma \left( \left\vert \left\{ \psi _{k}u_{j}>\frac{(N-1)}{C}\kappa_{j+2}^{-1}\right\} \right\vert _{\mu_B
	}\right) ,  \label{super}
	\end{equation}%
	with the notation%
	\begin{equation}
	\Gamma \left( t\right) \equiv \frac{1}{\tilde{\Phi}_{0}^{-1}\left( \frac{1}{t}%
		\right) }.  \label{psi_inv}
	\end{equation}%
	First recall 
	\begin{equation*}
	||f||_{L^{\tilde{\Phi}_0}(X)}\equiv \inf \left\{ a:\int_{X}\tilde{\Phi}_0\left( 
	\frac{f}{a}\right) \leq 1\right\} ,
	\end{equation*}%
	and note 
	\begin{equation*}
	\int_{\{\psi _{j+1}u_{j+1}>0\}}\tilde{\Phi}_0\left( \frac{1}{a}\right) =\tilde{%
		\Phi}_0\left( \frac{1}{a}\right) \left\vert \left\{ \psi
	_{j+1}u_{j+1}>0\right\} \right\vert _{\mu }.
	\end{equation*}%
	Now take 
	\begin{equation*}
	a=\Gamma \left( \left\vert \left\{ \psi _{j+1}u_{j+1}>0\right\} \right\vert
	_{\mu }\right) \equiv \frac{1}{\tilde{\Phi}_0^{-1}\left( \frac{1}{\left\vert
			\left\{ \psi _{j+1}u_{j+1}>0\right\} \right\vert _{\mu }}\right) }
	\end{equation*}%
	which obviously satisfies 
	\begin{equation*}
	\int_{\left\{ \psi _{j+1}u_{j+1}>0\right\} }\tilde{\Phi}_0\left( \frac{1}{a}%
	\right) d\mu =1.
	\end{equation*}%
	This gives 
	\begin{equation*}
	||1||_{L^{\tilde{\Phi}_0}\left( \left\{ \psi _{j+1}u_{j+1}>0\right\} d\mu
		\right) }\leq a=\Gamma \left( \left\vert \left\{ \psi
	_{j+1}u_{j+1}>0\right\} \right\vert _{\mu }\right) ,
	\end{equation*}%
	and to conclude (\ref{super}) we only need to observe that 
	\begin{equation*}
	\left\{ \psi _{j+1}u_{j+1}>0\right\} \subset \left\{ \psi _{j+1}u_{j}>\frac{\tau\psi_{j+1} \left\Vert \phi \right\Vert _{X}(N-1)}{C\kappa_{j+2}}\right\},
	\end{equation*}%
	which follows from (\ref{uk_support}) and (\ref{c est}).
	Next we use Chebyshev's inequality to obtain 
	\begin{equation}
	\left\vert \left\{ \psi _{j+1}u_{j}>\frac{c\psi_{j+1}\tau \left\Vert \phi \right\Vert _{X}(N-1)}{\kappa_{j+2}}\right\} \right\vert _{\mu }
	\leq \gamma \kappa_{j+2}^{2}\int \left( \psi _{j}u_{j}\right) ^{2}d\mu ,
	\label{cheb}
	\end{equation}%
	where $\gamma =\frac{c}{\tau ^{2}\left\Vert \phi \right\Vert
		_{X}^{2}(N-1)^{2}}$. 	
	Combining (\ref{hol})-(\ref{cheb}) we obtain 
	\begin{align*}
	\int &\left( \psi _{j+1}u_{j+1}\right) ^{2}d\mu_{B}  \\
	&\leq C||\left( \psi
	_{j+1}u_{j+1}\right) ^{2}||_{L^{\Phi }_0\left( B_{j};\mu_{B}  \right) }\cdot
	||1||_{L^{\tilde{\Phi}_0}\left( \{\psi _{j+1}u_{j+1}>0\};\mu_{B}  \right) } \\
	&\leq C\left[\frac{\varphi (r)}{r}\right]^2\kappa_{j+2}^{2}\int_{B_{j}}u_{j}^{2}d\mu_{B}  \cdot \Gamma \left(
	\left\vert \left\{ \psi _{j+1}u_{j}>\frac{\tau\psi _{j+1} \left\Vert \phi \right\Vert _{X}(N-1)}{C\kappa_{j+2}}\right\} \right\vert _{\mu_{B}  }\right)  \\
	&\leq C\left[\frac{\varphi (r)}{r}\right]^2\kappa_{j+2}^{2}\left( \int_{B_{j}}u_{j}^{2}d\mu_{B}  \right) \Gamma \left( \gamma \kappa_{j+2}^{2}\int \left( \psi _{j}u_{j}\right) ^{2}d\mu_{B} 
	\right),
	\end{align*}%
	or in terms of the quantities $U_{j}$,%
	\begin{equation}
	U_{j+1} \leq C\left[\frac{\varphi (r)}{r}\right]^2\kappa_{j+2}^{2}U_{j}\Gamma \left( \gamma \kappa_{j+2}^{2}U_{j}\right)   \label{iter} 
	=C\left[\frac{\varphi (r)}{r}\right]^2\kappa_{j+2}^{2}U_{j}%
	\frac{1}{\tilde{\Phi}_{0}^{-1}\left( \frac{1}{\gamma\kappa_{j+2}^{2}U_{j}}\right) }. 
	\end{equation}	
	Now we use the estimate (\ref{use}) on $\frac{1}{\tilde{\Phi}_0^{-1}\left( 
		\frac{1}{x}\right) }$ to determine the values of $N$ for
	which DeGiorgi iteration provides local boundedness of weak subsolutions,
	i.e. for which $U_{j}\rightarrow 0$ as $j\rightarrow \infty $ provided $U_{0}
	$ is sufficiently small. From (\ref{use}) and (\ref{iter}) we have 
	\begin{equation*}
	U_{j+1}\leq C\left[\frac{\varphi (r)}{r}\right]^2\frac{\kappa_{j+2}^{2}U_{j}}{\left( \ln\frac{1}{\gamma\kappa_{j+2}^{2}U_{j}}\right) ^{2}\dots \left( \ln^{(k-1)}\frac{1}{\gamma\kappa_{j+2}^{2}U_{j}}\right) ^{2}\left( \ln^{(k)}\frac{1}{\gamma\kappa_{j+2}^{2}U_{j}}\right) ^{2N}},
	\end{equation*}%
	provided
	\begin{equation*}
	\gamma\kappa_{j+2}^{2}U_{j}\leq\left[\exp^{(k)}(2N)\left(\exp^{(k-1)}(2N)\right)^2\dots \left(\exp(2N)\right)^2(2N)^2\right]^{-1},
	\end{equation*}%
	and using the notation 
	\begin{equation*}
	b_{j}\equiv \ln \frac{1}{U_{j}},
	\end{equation*}%
	we can rewrite this as 
	\begin{equation}\label{bj-ind}
	\begin{split}
	b_{j+1} \geq& b_{j}-2\ln\kappa_{j+2}  -2\ln \left( C\frac{\varphi (r)}{r}\right)\\
	&+2\ln \left( b_{j}-2\ln\kappa_{j+2} -\ln \gamma\right)+\dots+ 2\ln^{k-1} \left( b_{j}-2\ln\kappa_{j+2} -\ln \gamma\right)+2N\ln^{(k)} \left( b_{j}-2\ln\kappa_{j+2} -\ln \gamma\right),
	\end{split}
	\end{equation}%
	for $j\geq 0$, provided
	\begin{equation*}
	\frac{1}{\gamma \kappa_{j+2}^{2}}\frac{1}{U_{j}}%
	\geq \exp^{(k)}(2N)\left(\exp^{(k-1)}(2N)\right)^2\dots \left(\exp(2N)\right)^2(2N)^2\equiv M_{k,N};
	\end{equation*}%
	or, equivalently,
	\begin{equation}
	b_{j}>2\ln\kappa_{j+2} +\ln \gamma +\ln M_{k,N},\ \
	\ \ \ j\geq 0.  \label{prov}
	\end{equation}%
	We now use induction to show that\\
	\textbf{Claim}: Both (\ref{prov}) and 
	\begin{equation}
	b_{j}\geq b_{0}+2j,\ \ \ \ \ j\geq 0,  \label{bk}
	\end{equation}
	hold for $b_{0}$ taken sufficiently large depending on $N>1$, $k\geq 1$ and $	0<r<R $.
	
	First note that both (\ref{prov}) and (\ref{bk}) are trivial if $j=0$ and $b_{0}$ is
	large enough. Moreover, from the definition of $\kappa_j$, (\ref{def:kappa}), both claims are true for $j\leq E$ provided $b_{0}$ is sufficiently large depending on $N>1$, $k\geq 1$ and $	0<r<R $.
	Assume now that the claim is true for some $j\geq E$. Using the definition of $\kappa_j$, (\ref{def:kappa}), we therefore have for sufficiently large $b_0$ depending on $N$, $k$, $\tau$, and $\Vert\phi\Vert_{X}$
	\begin{equation*}
	\begin{split}
	b_{j+1} &\geq b_{0}+2j+2-2\ln\kappa_{j+2} - 2  -2\ln \left( C\frac{\varphi (r)}{r}\right)\\
	&\quad+2\ln \left(j+ (b_{0}+j-2\ln\kappa_{j+2} -\ln \gamma)\right)+\dots+  2\ln^{(k-1)} \left(j+ (b_{0}+j-2\ln\kappa_{j+2} -\ln \gamma)\right)\\
	&\quad+2N\ln^{(k)}\left(j+ (b_{0}+j-2\ln\kappa_{j+2} -\ln \gamma)\right)\\
	&\geq b_{0}+2(j+2)+2\ln(j+2)+\dots+2\ln^{(k-1)}(j+2)+2N\ln^{(k)}(j+b_0-\ln\gamma-C_{k,N})\\
	&\quad-2\ln(j+2)-\dots-2\ln^{(k-1)}(j+2)-(N+1)\ln^{(k)}(j+2)-2  -2\ln \left( C\frac{\varphi (r)}{r}\right)\\
	&\geq  b_{0}+2(j+1)+(N-1)\ln^{(k)}(j+b_0-\ln\gamma-C_{k,N})-2  -2\ln \left( C\frac{\varphi (r)}{r}\right),
	\end{split}
	\end{equation*}
	where $C_{k,N}>0$ is a constant depending only on $k$ and $N$.
	We have 
	\begin{equation*}
	(N-1)\ln^{(k)}(j+b_0-\ln\gamma-C_{k,N})-2  -2\ln \left( C\frac{\varphi (r)}{r}\right) \rightarrow \infty \quad\text{as}\  j\to\infty,
	\end{equation*}%
	and therefore for $b_{0}$ sufficiently large depending on $N$, $k$, $\gamma$, and $r$, we
	obtain 
	\begin{equation*}
	(N-1)\ln^{(k)}(j+b_0-\ln\gamma-C_{k,N})-2  -2\ln \left( C\frac{\varphi (r)}{r}\right) \geq 0,
	\end{equation*}%
	for all $j\geq E$, which gives (\ref{bk}) for $j+1$,%
	\begin{equation*}
	b_{j+1}\geq b_{0}+2(j+1),
	\end{equation*}%
	and also (\ref{prov}) for $j+1$, 
	\begin{eqnarray*}
		b_{j+1} \geq b_{0}+2j+2\geq 2\ln\kappa_{j+2} +\ln \gamma +\ln M_{k,N},
	\end{eqnarray*}
	where $\ln M_{k,N}=\exp^{(k-1)}(2N)+2\exp^{(k-2)}(2N)+\dots+4N+2\ln(2N)$.
	
	We note that it is sufficient to require for a constant $C_{k,N}$ large enough
	\[
	b_0\geq \ln \gamma+C_{k,N},\quad \text{and}\quad  (N-1)\ln^{(k)}(b_0+E-\ln\gamma-C_{k,N})\geq 2  +2\ln \left( C\frac{\varphi (r)}{r}\right)
	\]
	or
	\begin{equation}	\label{b_0_cond}
	\begin{split}
	b_{0}&\geq \exp^{(k-1)}\left(A_{k,N}\left( \frac{\varphi (r)}{r}\right) ^{\frac{2}{N-1}}\right)+\ln (e^{C_{k,N}}\gamma)\\
	&=\exp^{(k-1)}\left(A_{k,N}\left( \frac{\varphi (r)}{r}\right) ^{\frac{2}{N-1}}\right)+
	\ln \frac{C^2}{\tau ^{2}\left\Vert \phi \right\Vert _{X}^{2}(N-1)^{2}}
	\end{split}
	\end{equation}%
	for $C$ and $A_{k,N}$ sufficiently large depending on $k$ and $N$. This completes the proof of
	the induction step and therefore $b_{j}\rightarrow \infty $ as $j\rightarrow
	\infty $, or $U_{j}\rightarrow 0$ as $j\rightarrow \infty $, provided $%
	U_{0}=e^{-b_{0}}$ is sufficiently small.
	Altogether, we have shown that 
	\begin{equation*}
	u_{\infty }=(u-\tau \left\Vert \phi \right\Vert _{X})_{+}=0\quad \text{on}\
	B_{\infty }=B\left( 0,\frac{r}{2}\right) =\frac{1}{2}B\left( 0,r\right) ,
	\end{equation*}%
	and thus that 
	\begin{equation}
	u\leq \tau \left\Vert \phi \right\Vert _{X}\quad \text{on}\ B_{\infty }\ ,
	\label{bddness_temp}
	\end{equation}%
	provided $U_{0}=\int_{B}|u_{0}|^{2}d\mu_{B} =\frac{1}{w(B) }%
	\int_{B}u_{+}^{2}d\mu$ is sufficiently small. From (\ref{b_0_cond}) it follows
	that it is sufficient to require 
	\begin{align}\label{def delta}
\sqrt{\frac{1}{w(B)}\int_{B}|u_{+}|^{2}d\mu}=\sqrt{U_0}=e^{-b_{0}/2}&\leq
	C\tau \left\Vert \phi \right\Vert _{X}(N-1)\exp\left[-\frac{1}{2}\exp^{(k-1)} \left(
	A_{k,N}\left( \frac{\varphi (r)}{r}\right) ^{\frac{2}{N-1}}\right)\right] \\
	&=\eta
	_{k,N}(r)\tau \left\Vert \phi \right\Vert _{X}\notag\ ,  
	\end{align}%
	where $C$ and $A_{k,N}$ are the constant in (\ref{b_0_cond}), and%
	\begin{equation*}
	\eta _{k,N}(r)\equiv C(N-1)\exp\left[-\frac{1}{2}\exp^{(k-1)} \left(
	A_{k,N}\left( \frac{\varphi (r)}{r}\right) ^{\frac{2}{N-1}}\right)\right] .
	\end{equation*}
	
	To recover the general case we now consider two cases, $\sqrt{U_0}\leq \eta _{k,N}(r)\left\Vert
	\phi \right\Vert _{X}$ and $\sqrt{U_0}>\eta _{N}(r)\left\Vert \phi \right\Vert _{X}$. In the
	first case we obtain that (\ref{def delta}) holds with $\tau =1$ and thus 
	\begin{equation*}
	||u_{+}||_{L^{\infty }(\frac{1}{2}B)}\leq \left\Vert \phi \right\Vert _{X}\ .
	\end{equation*}%
	In the second case, when $\sqrt{\frac{1}{w(B) }%
		\int_{B}u_{+}^{2}}d\mu>\eta _{k,N}(r)\left\Vert \phi \right\Vert _{X}$, we let $%
	\tau =\frac{\sqrt{\frac{1}{w(B)}\int_{B}u_{+}^{2}d\mu}}{%
		\eta _{kN}(r)\left\Vert \phi \right\Vert _{X}}>1$ so that (\ref{def delta})
	holds, and then from (\ref{bddness_temp}) we get 
	\begin{equation*}
	||u_{+}||_{L^{\infty }(\frac{1}{2}B)}\leq \frac{\sqrt{\frac{1}{w(B) }\int_{B}u_{+}^{2}d\mu}}{\eta _{k,N}(r)}.
	\end{equation*}%
	Altogether this gives 
	\begin{equation*}
	||u_{+}||_{L^{\infty }(\frac{1}{2}B)}\leq \left\Vert \phi \right\Vert
	_{X}+A_{k,N}(r)\sqrt{\frac{1}{w(B)}\int_{B}u_{+}^{2}d\mu},
	\end{equation*}%
	where 
	\begin{equation*}
	A_{k,N}\left( r\right) =\frac{1}{\eta _{k,N}\left( r\right) }=C_{1}\left[\exp^{(k)} \left(
	C_{2}\left( \frac{\varphi \left( r\right) }{r}\right) ^{\frac{2}{N-1}}\right)\right]^{\frac{1}{2}} ,
	\end{equation*}%
	and where the constants $C_{1}$ and $C_{2}$ depend on $N$ and $k$, but not
	on $r$.
\end{proof}
Just as in \cite{KRSSh} we can obtain the following result
\begin{corollary}[of the proof.]
	Let all the conditions of Proposition \ref{DG} be satisfied. Then
	\begin{equation}
	\left\Vert u_{+}\right\Vert _{L_{w}^{\infty }(\frac{1}{2}B)}\leq A_{k,N}(3r)\left( 
	\frac{1}{w(3B) }\int_{B}u_{+}^{2}w\right) ^{\frac{1}{2}%
	}+\left\Vert \phi\right\Vert _{X},  \label{Inner ball 3}
	\end{equation}%
	with $A_{k,N}(r)$ defined in (\ref{def AN}).\label{cor:inner}
\end{corollary}

\subsection{Continuity}

We start with an auxiliary lemma.

\begin{lemma}\cite[Lemma 31]{KRSSh}
	\label{DeG Lemma}Suppose that the proportional vanishing Poincar\'{e}
	inequality (\ref{1-1poinc}) holds. Let $v\in W^{1,2}_{A}(\Omega)$, and define $\overline{v}\left( y\right) \equiv \max \left\{ 0,\min \left\{
	1,v\left( y\right) \right\} \right\} $. Set 
	\begin{eqnarray*}
		\mathcal{A} &\equiv &\left\{ y\in B\left( x,r\right) :\overline{v}\left( y\right) =
		0\right\} , \\
		\mathcal{C} &\equiv &\left\{ y\in B\left( x,r\right) :\overline{v}\left( y\right) =
		1\right\} , \\
		\mathcal{D} &\equiv &\left\{ y\in B\left( x,2r\right) :0<\overline{v}\left( y\right)
		<1\right\} .
	\end{eqnarray*}%
	Fix $x$ and $r$ and suppose that $\overline{v}$
	satisfies $\int_{B\left( x,3r/2\right) }\left\vert \nabla_{A} \overline{v}\left(
	y\right) \right\vert ^{2}dy\leq \frac{C_{0}}{r^2}w(\mathcal{D})$.
	Then if $w( \mathcal{A}) \geq \frac{1}{2}w(
	B\left( x,r\right)) $, we have 
	\begin{equation}
	C_{0}w(\mathcal{D}) \geq C_{1}\frac{w(\mathcal{C})^{2}}{w(B(x,r))} ,  \label{isoperimetric}
	\end{equation}
	where $C_1>0$ depends only on the constant in (\ref{1-1poinc}).
\end{lemma}
\begin{proof}
	Using H\"{o}lder inequality and (\ref{1-1poinc}), we have
	\begin{align}\label{vw_est}
	w(\mathcal{C})&=\int_{\mathcal{C}}\overline{v}(x)w(x)dx\leq \int_{B}\overline{v}(x)w(x)dx\leq \left(\int_{B}\overline{v}^2(x)w(x)dx\right)^{\frac{1}{2}}w(B)^{\frac{1}{2}}\notag\\
	&\leq Cr \left(\int_{\frac{3}{2}B}|\nabla_{A}\overline{v}(x)|^{2}dx\right)^{\frac{1}{2}}w(B)^{\frac{1}{2}}\leq C\sqrt{C_0}w(\mathcal{D})^{\frac{1}{2}}w(B)^{\frac{1}{2}},
	\end{align}	
	where we used the condition
	\[
	\int_{\frac{3}{2}B}|\nabla_{A}\overline{v}(x)|^{2}dx\leq \frac{C_{0}}{r^2}w(\mathcal{D}).
	\]
	Squaring both sides and rearranging gives (\ref{isoperimetric}).
\end{proof}

To prove our main continuity result, Theorem \ref{thm:continuity}, we will use the following weighted analog of Proposition 32 from \cite{KRSSh}. The proof follows closely that of \cite[Proposition 32]{KRSSh}, but modifications need to be made to account for the $L^2$ nature of the assumptions.
\begin{prop}\cite[Proposition 32]{KRSSh}
	\label{Holder thm}Suppose that $\Omega \subset \mathbb{R}^{n}$ is a domain
	in $\mathbb{R}^{n}$ with $n\geq 2$. Suppose the proportional vanishing
	Poincar\'{e} inequality (\ref{1-1poinc}) holds. For every ball $%
	B_{r_{0}}=B\left( x,r_{0}\right) $ contained in $\Omega $, assume that the
	local boundedness inequality,%
	\begin{equation}
	\left\Vert u_{+}\right\Vert _{L^{\infty }\left( B_{\frac{r}{2}}\right) }\leq 
	\frac{1}{2\sqrt{\delta \left( r\right) }}\sqrt{\frac{1}{w(B_{3r})}\int_{B_{r}}u_{+}^{2}w}\ ,  \label{LB}
	\end{equation}%
	holds for $Lu=0$ in $B_{3r}$ with $\delta \left( r\right) $%
	\thinspace $>0$, for all $0<r<r_{0}$. Let $r_{j}=\frac{r_{0}}{4^{j}}$, and
	set $B_{j}\equiv B_{r_{j}}=B\left( x,r_{j}\right) $. Suppose furthermore
	that the following summability condition holds,%
	\begin{equation}\label{summability}
	\sum_{j=1}^{\infty }\lambda _{j}=\infty ,\quad\text{where}\  \lambda _{j}\equiv \frac{1}{2^{3+\frac{C_{3}}{\delta ^{2}(r_{j})}}},
	\end{equation}%
	and $B_{j}=B(0,r_{j})=B\left( 0,\frac{r_{0}}{4^{j}}\right) $, and where $%
	C_{3}$ is a positive constant. Then if $u$ is a weak solution to $L
	u=0$ in $B_{r_{0}}$, we have $u\in C\left( B_{r_{0}}\right) $.
\end{prop}
\begin{proof}
	Let $u$ be as in the statement of the theorem, define
	\[
	v(x)\equiv \frac{2}{\mathrm{osc}_{B_r}u}\left\{u(x)-\frac{\sup_{B_r}u+\inf_{B_r}u}{2}\right\}.
	\]
	Then $v$ is a weak solution to $Lv=0$, and $-1\leq v\leq 1$ on $B_{3r}$. Moreover, assume $w\left(\{v\leq 0\}\cap B_{r}\right)\geq \frac{1}{2}w(B_r)$. We will show that $\mathrm{osc}_{B_{\frac{r}{2}}}v\leq 2-\lambda(r)$, where
	\[
	\lambda(r)\equiv \frac{1}{2^{3+\frac{C_{3}}{\delta ^{2}(r)}}}\in (0,1)
	\]
	with a constant $C_3$ depending only on the constant in (\ref{1-1poinc}).
	This will give
	\begin{equation}\label{eq:osc-control}
	\mathrm{osc}_{B_{\frac{r}{2}}}u\leq \left(1-\frac{\lambda(r)}{2}\right) \mathrm{osc}_{B_{r}}u.
	\end{equation}
	If instead we have $w\left(\{v\geq 0\}\cap B_{r}\right)\geq \frac{1}{2}w(B_r)$, the same argument can be applied to $-v$ to obtain (\ref{eq:osc-control}). Define
	\[
	v_k=2^k\left(v-(1-2^{-k})\right),
	\]
	and let $\psi$ be a smooth cutoff function supported in $B_{2r}$, and s.t. $\psi\equiv 1$ on $B_{3r/2}$. 
	We now would like to apply Lemma \ref{DeG Lemma} to $2v_k$. Let $\overline{v}_{k}\left( y\right) \equiv \max \left\{ 0,\min \left\{
	1,2v_k\left( y\right) \right\} \right\}$ and define
	\begin{eqnarray*}
		\mathcal{A}_k &\equiv &\left\{ y\in B\left( x,r\right) :\overline{v}_{k}\left( y\right) =
		0\right\} , \\
		\mathcal{C}_k &\equiv &\left\{ y\in B\left( x,r\right) :\overline{v}_{k}\left( y\right) =
		1\right\} , \\
		\mathcal{D}_k &\equiv &\left\{ y\in B\left( x,2r\right) :0<\overline{v}_{k}\left( y\right)
		<1\right\}.
	\end{eqnarray*}%
	One can verify that $L\overline{v}_{k}=0$, and moreover $\{y\in B\left( x,2r\right)\colon \nabla_{A}\overline{v}_{k}(y)\neq 0\}\subseteq \mathcal{D}$.  We now show a slightly improved version of the Caccioppoli inequality (\ref{Cacc_inhomog}) with $\phi=0$. 
	Since $L\overline{v}_{k}=0$ we have by H\"{o}lder's inequality
	\begin{align*}
	\int_{\mathcal{D}}\psi^{2}|\nabla_{A}\overline{v}_{k}|^{2}&=	\int_{B_{2r}}\psi^{2}|\nabla_{A}\overline{v}_{k}|^{2}=-2\int_{B_{2r}}\psi \overline{v}_{k}\nabla\psi\cdot A\nabla \overline{v}_{k}
	\leq 2\int_{\mathcal{D}}\left\vert\psi \overline{v}_{k}\nabla\psi\cdot A\nabla \overline{v}_{k}\right\vert\\
	&\leq \frac{1}{2}\int_{\mathcal{D}}\psi^{2}|\nabla_{A}\overline{v}_{k}|^{2}+C \int_{\mathcal{D}} \left\vert\nabla_{A}\psi\right\vert^2\overline{v}_{k}^{2},
	\end{align*}
	and thus
	\[
	\int_{B_{2r}}\psi^{2}|\nabla_{A}\overline{v}_{k}|^{2}\leq C \int_{\mathcal{D}} \frac{\left\vert\nabla_{A}\psi\right\vert^2}{w}\overline{v}_{k}^{2}w \leq C\left\Vert\frac{|\nabla_{A}\psi|^{2}}{w}\right\Vert_{L^{\infty}}w(\mathcal{D}),
	\]
	where for the last inequality we used the fact that $\overline{v}_{k}\leq 1$. This gives
	\[
	\int_{B_{2r}}\psi^{2}|\nabla_{A}\overline{v}_{+}|^{2}\leq C\left\Vert\frac{|\nabla_{A}\psi|^{2}}{w}\right\Vert_{L^{\infty}}w(\mathcal{D}).
	\]
	We are now set to apply Lemma \ref{DeG Lemma} to $\overline{v}_{k}$. 
	Since $v\leq 0$ implies $v_k\leq 0$ which is equivalent to $\overline{v}_{k}=0$, we have $w(\mathcal{A}_k)\geq \frac{1}{2}w(B_r)$ by our assumption on $v$. Lemma \ref{DeG Lemma} is now applied to $2v_k$ recursively for $k=0,1,2,\dots$ but only as long as
	\begin{equation}\label{aslong}
	\frac{1}{w(B_{3r})}\int_{B_{r}}(v_k)^{2}_{+}w\geq \delta(r),
	\end{equation}
	where $\delta(r)$ is as in (\ref{LB}). Since $v_{k+1}=2v_k-1$ we have
	\[
	w(\mathcal{C}_k)\geq \int_{\mathcal{C}_k}(v_{k+1})^{2}_{+}w\geq  \delta(r)w(B_{3r}),
	\]
	and therefore by Lemma \ref{DeG Lemma}
	\[
	w\left(\{0<2v_k<1\}\cap B_{2r}\right)=w(\mathcal{D}_k)\geq C\frac{\left(\delta(r)w(B_{3r})\right)^2}{w(B_{r})}\geq C\delta(r)^2w(B_{3r})=:\alpha(r).
	\]
	This gives
	\begin{align*}
	w(B_{2r})&\geq w\left(\{2v_k\leq 0\}\cap B_{2r}\right)= w\left(\{2v_{k-1}\leq 1\}\cap B_{2r}\right)\\
	&=w\left(\{2v_{k-1}\leq 0\}\cap B_{2r}\right)+w\left(\{0<2v_k\leq 1\}\cap B_{2r}\right)\\
	&\geq w\left(\{2v_{k-1}\leq 0\}\cap B_{2r}\right)+\alpha(r)\\
	&\dots\\
	&\geq w\left(\{2v_{0}\leq 0\}\cap B_{2r}\right)+k\alpha(r)\geq k\alpha(r).
	\end{align*}
	Since $\alpha(r)>0$, this inequality must fail for a finite $k$, namely a unique integer $k_0$ satisfying $\frac{w(B_{2r})}{\alpha(r)}<k_0\leq \frac{w(B_{2r})}{\alpha(r)}+1$. Thus condition (\ref{aslong}) must fail for $k=k_0+1$, so
	\[
	\int_{B_{r}}(v_k)^{2}_{+}w< \delta(r)w(B_{3r}).
	\]
	By assumption (\ref{LB}) this gives that in $B_{\frac{r}{2}}$ we have
	\[
	v_{k_0+1}<\frac{1}{2},
	\]
	which gives 
	\[
	v\leq 1-\frac{1}{2^{k_0+2}}\quad\text{in}\  B_{\frac{r}{2}}.
	\]
	Recall that by definition $\alpha(r)=C\delta(r)^2w(B_{3r})$, and thus 
	\[
	k_0\leq \frac{w(B_{2r})}{\alpha(r)}+1\leq \frac{C_3}{\delta(r)^2}+1,
	\]
	so
	\[
	\frac{1}{2^{k_0+2}}\geq  \frac{1}{2^{3+\frac{C_{3}}{\delta ^{2}(r)}}}=\lambda(r),\quad\text{and}\  \sup_{B_{\frac{r}{2}}}v\leq 1-\lambda(r).
	\]
	Using the definition of $v$ we obtain (\ref{eq:osc-control}) for all $r\in (0,r_0)$. Iterating this inequality gives
	\[
	\mathrm{osc}_{B_{\frac{r}{4^l}}}u\leq \left\{\prod_{j=1}^{l}\left(1-\frac{\lambda\left(\frac{r}{4^j}\right)}{2}\right) \right\}\mathrm{osc}_{B_{r}}u,\quad l\geq 1.
	\]
	By summability condition (\ref{summability}) the right hand side converges to zero as $l\to\infty$, which implies continuity of $u$ at $x$.
\end{proof}

We now prove Theorem \ref{thm:continuity}

\begin{proof}
	Using Corollary \ref{cor:inner} it is sufficient to show that the summability condition (\ref{summability}) is satisfied with $\delta(r)=4A_{k,N}(3r)^{-2}$, and $A_{k,N}$ defined in (\ref{def AN}). Proposition \ref{Holder thm} will then imply continuity of weak solutions. We proceed as in the proof of Theorem 84 in \cite{KRSSh}. First, using (\ref{phi-cond}) we have
	\[
	\delta(r)=\frac{1}{4A_{k,N}(3r)^2}=\frac{1}{4C_{1}^{2}}\left[\exp^{(k)} \left(
	C_{2}\left( \frac{\varphi \left( 3r\right) }{3r}\right) ^{\frac{2}{N-1}}\right)\right]^{-1}\geq \left[\exp^{(k)} \left(
	C'\left[\ln^{(k+2)}\left( \frac{1 }{3r}\right) \right]^{\frac{2\alpha N}{N-1}}\right)\right]^{-1}.
	\]
	We now use the fact that $\frac{2\alpha N}{N-1}<1$ to conclude that for every $\ve>0$ there exists $a_0$ sufficiently large, depending on $\ve$, $N$, and $\alpha$, satisfying
	\[
	\left[\ln^{(k+2)}a\right]^{\frac{2\alpha N}{N-1}}\leq \ve \ln^{(k-1)}\left(\ve \ln^{(3)}a\right)\quad \forall a\geq a_0.
	\]
	This gives for $\ve>0$ and $r>0$ sufficiently small
	\[
	\delta^{2}(r)\geq  \left[\exp^{(k)} \left(
	C'\ve\ln^{(k-1)}\left(\ve\ln^{(3)} \frac{1 }{3r}\right)\right)\right]^{-2}\geq \left[\exp\left(\ve\ln^{(3)}\frac{1}{3r}\right)\right]^{-2}=\left[\ln^{(2)}\frac{1}{3r}\right]^{-2\ve}.
	\]
	Now, provided $2\ve<1$ we have for every $\gamma<1$
	\[
	\left[\ln^{(2)}a\right]^{2\ve}\leq \gamma\ln^{(2)}a,
	\]
	for $a$ sufficiently large.
	We therefore conclude for $j$ sufficiently large and $\gamma$ chosen sufficiently small
	\[
	\lambda_j=\frac{1}{2^{3+\frac{C_{3}}{\delta ^{2}(r_{j})}}}\gtrsim \exp\left(-C''\left[\ln^{(2)}\frac{1}{3r_j}\right]^{2\ve}\right)\geq\exp\left(-C''\gamma\ln^{(2)}\frac{1}{3r_j}\right)
	\gtrsim\exp\left(-C'''\gamma\ln j\right)\gtrsim\frac{1}{j}.
	\]
	This gives (\ref{summability}) and by Proposition \ref{Holder thm} we conclude continuity of weak solutions.
\end{proof}

\section{Sufficient conditions}
\subsection{$(2,2)$ Poincar\'{e} implies proportional vanishing Poincar\'{e}}\label{sec:poinc}
In this section we show that the standard weighted $(2,2)$ Poincar\'{e} inequality
\begin{equation}\label{reg-poinc}
\int_{B}|v-v_B|^2 d\mu\leq Cr\int_{\frac{3}{2}B}\left\vert \nabla_{A} v\right\vert^2\,dx,\ \ \ \ \ v\in Lip\left( \frac{3}{2}B\right) 
\end{equation}
where $v_B=\frac{1}{w(B)}\int_{B}vd\mu$, implies the proportional vanishing Poincar\'{e} inequality (\ref{1-1poinc}). Let $v\in Lip\left( \frac{3}{2}B\right)$ and suppose it vanishes on a subset $E\subseteq B$ satisfying $w(E)\geq \frac{1}{2}w(B)$. We have
\begin{align*}
\left\Vert v\right\Vert_{L^2(B;\mu)}	&= \left\Vert v\right\Vert_{L^2(B\setminus E;\mu)}    = \left\Vert v-\frac{1}{w(B)}\int_{E}vd\mu\right\Vert_{L^2(B\setminus E;\mu)}\\
&\leq \left\Vert v-\frac{1}{w(B)}\int_{E}v-\frac{1}{w(B)}\int_{B\setminus E}vd\mu\right\Vert_{L^2(B\setminus E;\mu)}
+\frac{w(B\setminus E)}{w(B)} \left\Vert v\right\Vert_{L^2(B\setminus E;\mu)}\\
&\leq \left\Vert v-\frac{1}{w(B)}\int_{B}vd\mu\right\Vert_{L^2(B;\mu)}+\frac{1}{2} \left\Vert v\right\Vert_{L^2(B\setminus E;\mu)},
\end{align*}
which gives using (\ref{reg-poinc})
\[
\left\Vert v\right\Vert_{L^2(B;\mu)}= \left\Vert v\right\Vert_{L^2(B\setminus E;\mu)} \leq 2\left\Vert v-\frac{1}{w(B)}\int_{B}vd\mu\right\Vert_{L^2(B;\mu)}\leq 2Cr\left\Vert \nabla_{A} v\right\Vert_{L^2\left(\frac{3}{2}B\right)}.
\]

\subsection{$(\Psi,1)$ Orlicz Sobolev implies $(\Phi,2)$ Orlicz Sobolev}\label{sec:orl}

Assume that the following $(\Psi,1)$ Orlicz Sobolev holds
\begin{equation}\label{orl-sob-w-1}
\left\Vert v\right\Vert _{L^{\Psi }_{\mu_{B}}}\leq C\frac{\varphi \left(
	r\left( B\right) \right) }{w(B)}\left\Vert \nabla_{A} v\sqrt{w}\right\Vert _{L^{1} },\ \ \ \ \ v\in Lip_{c}\left( B\right) ,
\end{equation}%
for the Young function $\Psi\equiv \sqrt{\Phi_{k,N}}$ with $N>1$ and $k\in \N$, where $\Phi_{k,N}$ is defined in (\ref{phi-def}). We claim that the $(\Phi,2)$ Orlicz Sobolev inequality (\ref{orl-sob-w}) holds with Young function $\Phi\equiv\Phi_{k,N}=\Psi^{2}$, namely
\begin{equation}\label{orl-sob-w-again}
\left\Vert v\right\Vert _{L^{\Phi }_{\mu_{B}}}\leq C\frac{\varphi \left(
	r\left( B\right) \right) }{w(B)^{\frac{1}{2}}}\left\Vert \nabla_{A} v\right\Vert _{L^{2} },\ \ \ \ \ v\in Lip_{c}\left( B\right) .
\end{equation}%
First, let $H$ be a Young function satisfying $H(t)=t^{2}\cdot\ln t\cdot\dots \cdot(\ln^{(k)} t)^N$ for large $t$, and extended continuously to a quadratic function for small values of $t$, similarly to $\Phi_{k, N}$ in (\ref{phi-def}).
Note that we have
\[
\Psi\circ H\approx \Phi.
\]
Indeed, $\ln H(t)=2\ln t+\ln \ln t+\dots+N\ln^{(k+1)} t\approx \ln t$, so that for $t^2\geq E$ we have
\begin{align*}
\Psi(H(t))&=t^{2}\cdot\ln t\cdot\ln H^2(t)\cdot\dots \cdot(\ln^{(k)} t)^N\cdot (\ln^{(k)} H^2(t))^N\\
&\approx t^{2}\cdot\ln t\cdot\ln H(t)\cdot\dots \cdot(\ln^{(k)} t)^N\cdot (\ln^{(k)} H(t))^N\\
&\approx t^{2}\cdot(\ln t)^2\cdot\dots \cdot(\ln^{(k)} t)^{2N}\\
&\approx t^{2}\cdot(\ln t^2)^2\cdot\dots \cdot(\ln^{(k)} t^2)^{2N}=\Phi(t).
\end{align*}
Now let  $v\in Lip_{c}\left( B\right)$, then $ H(v)\in Lip_{c}\left( B\right)$, and
\[
\nabla_{A}H(v)=H'(v)\nabla_{A}v.
\]
We now apply the $(\Psi,1)$ Orlicz Sobolev inequality (\ref{orl-sob-w-1}) to $H(v)$ and use H\"{o}lder's inequality to obtain
\begin{align}\label{inter}
\left\Vert H(v)\right\Vert _{L^{\Psi }_{\mu_{B}}}&\leq C\frac{\varphi \left(r\left( B\right) \right) }{w(B)}\left\Vert \nabla_{A} H(v)\sqrt{w}\right\Vert _{L^{1} }=C\frac{\varphi \left(r\left( B\right) \right) }{w(B)}\int |H'(v)||\nabla_{A}v|\sqrt{w}\notag\\
&\leq C\frac{\varphi \left(r\left( B\right) \right) }{w(B)}\left(\int |H'(v)|^2w\right)^{\frac{1}{2}}\left(\int |\nabla_{A}v|^2\right)^{\frac{1}{2}}=C\frac{\varphi \left(r\left( B\right) \right)}{w(B)^{\frac{1}{2}}} \left(\int |H'(v)|^2 d\mu_{B}\right)^{\frac{1}{2}}\left(\int |\nabla_{A}v|^2\right)^{\frac{1}{2}},
\end{align}
where we used the notation $\displaystyle d\mu_B\equiv\frac{w(x)dx }{w(B)}$. Differentiating $H(t)$ we obtain
\[
H'(t)\approx 2t\cdot\ln t\cdot\dots \cdot(\ln^{(k)} t)^N\approx \frac{H(t)}{t},
\]
so that
\[
|H'(t)|^{2}\approx t^{2}\cdot(\ln t)^2\cdot\dots \cdot(\ln^{(k)} t)^{2N}\approx \Phi(t).
\]
Combining with (\ref{inter}) gives
\[
\left\Vert H(v)\right\Vert _{L^{\Psi }_{\mu_{B}}}\leq C\frac{\varphi \left(r\left( B\right) \right)}{w(B)^{\frac{1}{2}}}\left(\int \Phi(v)d\mu_B\right)^{\frac{1}{2}}\left(\int |\nabla_{A}v|^2\right)^{\frac{1}{2}}
\]
Next, we note that $\Psi(t)=\sqrt{\Phi_0(t^2)}$, where $\Phi_0$ is defined by (\ref{phi0-def}). As shown in Section \ref{sec:fam}, $\Phi_0=\Phi_{k,N}^{0}$ is submultiplicative, and therefore $\Psi$ is submultiplicative. By definition of the norm (\ref{norm_def}) this implies
\[
\Psi^{(-1)}\left(\int \Psi\left(H(v)\right) d\mu_{B}\right)\leq 	\left\Vert H(v)\right\Vert _{L^{\Psi }_{\mu_{B}}},
\]
which gives using $\Psi\circ H\approx \Phi$
\begin{equation}\label{inter2}
\Psi^{(-1)}\left(\int \Phi\left(v\right) d\mu_{B}\right)\lesssim C\frac{\varphi \left(r\left( B\right) \right)}{w(B)^{\frac{1}{2}}}\left(\int \Phi(v)d\mu_{B}\right)^{\frac{1}{2}}\left(\int |\nabla_{A}v|^2\right)^{\frac{1}{2}}
=C\varphi(r(B)).
\end{equation}
We now need an estimate on $\Psi^{-1}(t)/\sqrt{t}$, namely, we would like to show
\[
\frac{\Psi^{-1}(t)}{\sqrt{t}}\gtrsim \Phi^{(-1)}(t). 
\]
Let $s=\Psi(t)$, then for $t^2\geq E$ we have
\begin{align*}
s&=t\ln t^2\ln\ln t^2\dots(\ln^{(k)}t^2)^{N}\\
\ln s&=\ln t+\ln\ln t^2+\dots+N\ln^{(k+1)}t^2\approx \ln t\\
t&=\frac{s}{\ln t^2\ln\ln t^2\dots(\ln^{(k)}t^2)^{N}}\approx \frac{s}{\ln s\ln\ln s\dots(\ln^{(k)}s)^{N}},
\end{align*}
and therefore
\[
\frac{\Psi^{-1}(t)}{\sqrt{t}}\approx \frac{\sqrt{t}}{\ln t\ln\ln t\dots(\ln^{(k)}t)^{N}}.
\]
Similarly we let $s=\Phi(t)$ and estimate
\begin{align*}
s&=t^2(\ln t^2)^2(\ln\ln t^2)^2\dots(\ln^{(k)}t^2)^{2N}\\
\ln s&=2\ln t+2\ln\ln t^2+\dots+2N\ln^{(k+1)}t^2\approx \ln t\\
t&=\frac{\sqrt{s}}{\ln t^2\ln\ln t^2\dots(\ln^{(k)}t^2)^{N}}\approx \frac{\sqrt{s}}{\ln s\ln\ln s\dots(\ln^{(k)}s)^{N}},
\end{align*}
which gives
\[
\Phi^{(-1)}(t)\approx  \frac{\sqrt{t}}{\ln t\ln\ln t\dots(\ln^{(k)}t)^{N}}.
\]
Combining with (\ref{inter2}) we obtain
\begin{equation}\label{almost-sob}
\Phi^{(-1)}\left(\int \Phi\left(v\right) d\mu_{B}\right)\lesssim C\frac{\varphi \left(r\left( B\right) \right)}{w(B)^{\frac{1}{2}}}\left(\int |\nabla_{A}v|^2\right)^{\frac{1}{2}}.
\end{equation}
We now claim that (\ref{almost-sob}) implies (\ref{orl-sob-w-again}). The proof appears in \cite{KRSSh2}, but we repeat it here for the reader's convenience. Given $v\in Lip_{c}\left( B\left( x,r \right) \right)$ define 
\[
u\equiv \frac{v}{a},\  \  a=C_1\frac{\varphi \left(r\left( B\right) \right)}{w(B)^{\frac{1}{2}}} \
\left\Vert \nabla _{A}v\right\Vert _{L^{2}}
\]
for a constant $C_1\geq 0$ to be determined later. Applying (\ref{almost-sob}) to $u$ gives
\[
\int\Phi \left(
u\right) d\mu _{B }\leq \Phi\left(\frac{C\varphi \left( r \right) \left\Vert \nabla _{A}w\right\Vert _{L^{2} }}{a}\right)=\Phi\left(\frac{C}{C_1}\right)\leq 1,
\]
provided $C_1\geq \frac{C}{\Phi^{-1}(1)}$. Therefore
\[
\int\Phi \left(
\frac{v}{a}\right) d\mu _{B}\leq 1,
\]
and by the definition of the norm, 
\[
\left\Vert v\right\Vert _{L^{\Phi }\left( \mu_{B} \right) }\leq a=C\frac{\varphi \left(r\left( B\right) \right)}{w(B)^{\frac{1}{2}}}\
\left\Vert \nabla _{A}v\right\Vert _{L^{2} }.
\]

\subsection{Admissibility}\label{sec:adm}
Suppose that Orlicz Sobolev inequality (\ref{orl-sob-w}) holds in the ball $B=B(x,r)$ with the Young function $\Phi$ satisfying $\Phi(t)=\Phi_0(t^2)$ for some Young function $\Phi_0$, and suppose $\phi^2/w^2 \in L^{\tilde{\Phi}_0}_{\mu_B}$, where $\tilde{\Phi}_{0}$ is the dual of $\Phi_0$. Then $\phi$ is $A$-admissible at $(x,r)$ and
\[
\Vert \phi \Vert _{X\left( B\left( x,r \right) \right) }\leq C\varphi(r)\left\Vert \left(\frac{\phi}{w}\right)^{2}\right\Vert ^{\frac{1}{2}}_{L^{\tilde{\Phi}_{0} }\left( \mu_{B} \right) }.
\]
The proof is similar to the case of $(\Phi,1)$ Orlicz Sobolev inequality, see \cite{KRSSh2}. Let $\phi^2/w^2 \in L^{\tilde{\Phi}_0}_{\mu_B}$ and $v\in \left(W^{1,2}_{A}\right)_{0}(B)$, using H\"{o}lder's inequality for Orlicz norms we have
\begin{align*}
\left\vert\int_{B}v\phi\right\vert&=\left\vert\int_{B}v\phi\,\chi_{\{v\neq 0\}}\right\vert=\left\vert\int_{B}v\frac{\phi}{\sqrt{w}}\sqrt{w}\,\chi_{\{v\neq 0\}}\right\vert\\
&\leq \left(\int v^2\frac{\phi^2}{w}\right)^{\frac{1}{2}}w\left(\{v\neq 0\}\right)^{\frac{1}{2}}
=w(B)^{\frac{1}{2}}\left(\int_{B}v^2\frac{\phi^2}{w^2}\frac{w}{w(B)}\right)^{\frac{1}{2}}w\left(\{v\neq 0\}\right)^{\frac{1}{2}}\\
&\leq 2w(B)^{\frac{1}{2}}\left\Vert v^2\right\Vert^{\frac{1}{2}} _{L^{\Phi_0 }\left( \mu_{B} \right) }\left\Vert \left(\frac{\phi}{w}\right)^{2}\right\Vert ^{\frac{1}{2}}_{L^{\tilde{\Phi}_{0} }\left( \mu_{B} \right) }w\left(\{v\neq 0\}\right)^{\frac{1}{2}}\\
&\approx w(B)^{\frac{1}{2}}\left\Vert v\right\Vert _{L^{\Phi }\left( \mu_{B} \right) }\left\Vert \left(\frac{\phi}{w}\right)^{2}\right\Vert ^{\frac{1}{2}}_{L^{\tilde{\Phi}_{0} }\left( \mu_{B} \right) }w\left(\{v\neq 0\}\right)^{\frac{1}{2}},
\end{align*}
where for the last equality we used $\left\Vert v^2\right\Vert^{\frac{1}{2}} _{L^{\Phi_0 }\left( \mu_{B} \right) }\approx \left\Vert v\right\Vert _{L^{\Phi }\left( \mu_{B} \right) }$, see e.g. \cite[Lemma 18]{HLWK}.
Applying (\ref{orl-sob-w}) to the second factor on the right gives
\begin{align*}
\left\vert\int_{B}v\phi\right\vert&\leq C\varphi(r)\left\Vert \nabla_{A}v\right\Vert _{L^{2} }\left\Vert \left(\frac{\phi}{w}\right)^{2}\right\Vert ^{\frac{1}{2}}_{L^{\tilde{\Phi}_{0} }\left( \mu_{B} \right) }w\left(\{v\neq 0\}\right)^{\frac{1}{2}},
\end{align*}
so that
\[
\frac{\left\vert\int_{B}v\phi\right\vert}{w\left(\{v\neq 0\}\right)^{\frac{1}{2}}\left\Vert \nabla_{A}v\right\Vert _{L^{2} }}\leq 
C\varphi(r)\left\Vert \left(\frac{\phi}{w}\right)^{2}\right\Vert ^{\frac{1}{2}}_{L^{\tilde{\Phi}_{0} }\left( \mu_{B} \right) }
\]
which implies the claim.
\begin{remark}
	Recall that in the case of a uniformly elliptic operator we have $w\equiv 1$, $\nabla_{A}\approx\nabla$ and 
	\[
	\Phi(t)=t^{\frac{2n}{n-2}},\quad \Phi_0(t)=t^{\frac{n}{n-2}},\quad \tilde{\Phi}_{0}(t)=t^{\frac{n}{2}}.
	\]
	For the right hand side to be admissible, we therefore need $\phi\in L^{n}(B)$. However, for local boundedness and H\"{o}lder continuity of solutions it is actually sufficient to require $\phi\in L^{q}(B)$ with $q>n/2$, see e.g. \cite{GT}. In fact, in \cite{C-UR} local boundedness was shown under an even weaker assumption that $\phi \in L^{A}(\Omega)$ with $A(t)=t^{\frac{n}{2}}\ln (e+t)^{q}$, $q>n/2$.
\end{remark}

\section{Examples}\label{sec:ex}
In this section we give examples of operators satisfying sufficient conditions in Theorems \ref{DG} and \ref{thm:continuity}.
\subsection{Infinitely degenerate operator with one vanishing eigenvalue}
Consider the operator of the form (\ref{def:op}) with the matrix $A$ satisfying
\[
A(x)\approx \mathrm{diag}\{1,1,\dots,1, f^{2}(x_1)\}
\]
studied in \cite{KRSSh, KS, KRSSh2}. Let $L$ be given by (\ref{def:op}) with $A(x)$ be as above, and assume that $f(x)=e^{-F(x)}$ is even, nonnegative, and $F$ satisfies the structure assumptions from \cite[Chapter 7]{KRSSh}, in particular it is twice differentiable, decreasing, concave, and doubling on some interval $(0,R)$. Such matrix satisfies (\ref{A-cond}) with $k=1$ and $w\equiv 1$. Moreover, the $(\Psi,1)$ Orlicz-Sobolev inequality holds with $\Psi(t)=t(\ln t)^N$, $N>1$, for large $t$, provided $F(x)\leq 1/x^{\sigma}$, $\sigma<1$. From Section \ref{sec:poinc} it follows that the $(\Phi,2)$ Orlicz-Sobolev inequality holds with $\Phi=\Psi^{2}=\Phi_{1,N}$ as in (\ref{phi-def}), and therefore Theorems \ref{DG} and \ref{thm:continuity} hold, reproducing the results from \cite{KRSSh}.

\subsection{Degeneracy controlled by a $2$-admissible weight}
Another example that fits into our framework is a degeneracy controlled by a $2$-admissible weight, see e.g. \cite{FKS}. More precisely, let $L=\nabla\cdot A\nabla$ with the matrix $A$ satisfying
\begin{equation}\label{A-weighted}
\xi\cdot A(x)\xi\approx w(x)|\xi|^2,\quad \text{a.e.}\   x\in\Omega,\  \forall \xi\in\R^n.
\end{equation}
If $w$ is $2$-admissible, then one can employ the standard Moser or DeGiorgi iteration scheme to obtain local boundedness and H\"{o}lder continuity of weak solutions. By definition of a $2$-admissible weight \cite{BB11}, the following weighted Sobolev and Poincar\'{e} inequalities hold
\begin{align}
\left(\frac{1}{w(B)}\int_{B}|v-v_B|^2 d\mu\right)^{\frac{1}{2}}\leq Cr\left(\int_{B}|\nabla v|^{2}d\mu\right)^{\frac{1}{2}},\ \ \ \ \ v\in Lip\left( B\right)  \label{2-2-std},	\\
\left(\frac{1}{w(B)}\int_{B}|v|^{2\sigma} d\mu\right)^{\frac{1}{2\sigma}}\leq Cr\left(\int_{B}|\nabla v|^{2}d\mu\right)^{\frac{1}{2}},\ \ \ \ \ v\in Lip_{c}\left( B\right) \label{sob-std},	
\end{align}
where $\sigma>1$. From (\ref{A-weighted}) it immediately follows that (\ref{2-2-std}) implies (\ref{reg-poinc}), which in turn implies (\ref{1-1poinc}), see Section \ref{sec:poinc}. Moreover, since $\|v\|_{L^{2\sigma}_{\mu_B}}\leq \|v\|_{L^{\Phi_{k,N}}_{\mu_B}}$ for any $\sigma>1$, $k\in \N$, $N>1$, see e.g. \cite{RaoRen}, we have that (\ref{sob-std}) implies (\ref{orl-sob-w}) using condition (\ref{A-weighted}) again and the fact that $\varphi(r)\geq r$.

\subsection{General setup}
Our final example is the class of degenerate operators described in \cite{C-UR} (see also a recent preprint \cite{C-UMR}) with the matrix $A$ of the same form (\ref{A-cond}) considered here. One can easily check that assumptions in \cite{C-UR} on the operator $L$ satisfy our hypotheses (H1) and (H2). If $u$ is a weak solution of $Lu=\phi$ with an admissible $\phi$ then by Theorem \ref{DG} it is locally bounded, and we recover the result in \cite{C-UR}. Note however that while the assumptions on the operator in \cite{C-UR} are much stronger (for example, a $(2\sigma,2)$-Sobolev inequality is assumed with $\sigma>0$), the authors make a much weaker assumption on the right hand side $\phi$. In particular, $\phi$ is far from being admissible. Moreover, the authors together with S.F. MacDonald have recently generalized their result to the case when only an Orlicz-Sobolev inequality is assumed, with an Orlicz gain $\Phi$ similar to (\ref{phi-def}) with $k=1$ and $N>1/2$. Note however that in both papers \cite{C-UR} and \cite{C-UMR} the authors consider a Dirichlet problem so they have an additional assumption of the solution vanishing on the boundary of $\Omega$. On the other hand, they only need to assume a global Sobolev inequality to hold in the domain $\Omega$.

Finally, we point out a recent paper \cite{KMY} which established necessary and sufficient conditions for one-dimensional weighted Orlicz-Sobolev inequalities. It also gives an example of a weight for which a $(\sigma,1)$-Sobolev inequality does not hold for any $\sigma>1$, but a $(\Psi,1)$ Orlicz Sobolev inequality holds for $\Psi$ as in Section \ref{sec:orl}. Extending these results to $n$-dimensional weights would provide new non-trivial examples of operators satisfying (H1) and (H2) and thus give a powerful application of the abstract theory developed in this paper.

\end{document}